\documentclass{birkjour}
\usepackage{quaternions}
\makeindex
 \newtheorem{thm}{Theorem}[section]
 \newtheorem{cor}[thm]{Corollary}
 \newtheorem{lem}[thm]{Lemma}
 
 \theoremstyle{definition}
 \newtheorem{defn}[thm]{Definition}
 \theoremstyle{remark}
 \newtheorem{rem}[thm]{Remark}
 \newtheorem*{ex}{Example}
 \numberwithin{equation}{section}

 \usepackage{amsmath}
 \usepackage{amssymb}
 \usepackage{amsthm}
 \usepackage{mathrsfs}

 \newtheorem{bez}[thm]{Notation}

\def\R		{\mathbb{R}}
\def\C		{\mathbb{C}}
\def\N		{\mathbb{N}}
\def\O		{\mathcal{O}}
\def\d		{\:\mathrm{d}}

\def\L		{\mathscr{L}}
\def\Id		{\operatorname{Id}}
\def\F		{\mathscr{F}}
\renewcommand{\vec}[1]{\boldsymbol{#1}}

\def\G		{\mathcal{G}}
\def\j		{\vec{j}}
\def\k		{\vec{k}}

\def\e		{\vec{e}}

\def\c		{\vec{c}}
\def\A		{\vec{A}}
\def\B		{\vec{B}}
\def\x		{\vec{x}}
\def\y		{\vec{y}}

\def\u		{\vec{u}}
\def\v		{\vec{v}}

\begin{document}

%-------------------------------------------------------------------------
% editorial commands: to be inserted by the editorial office
%
%\firstpage{1} \volume{228} \Copyrightyear{2004} \DOI{003-0001}
%\seriesextraline{This is the Concrete Title of this Book\br H.E. R and S.T.C. W, Eds.}
%\firstpage{1}

\title{A General Geometric Fourier Transform}

\author[Bujack]{Roxana Bujack}
\address{%
Universit\"at Leipzig\\
Institut f\"ur Informatik\\
Johannisgasse 26\\
04103 Leipzig, Germany}
\email{bujack@informatik.uni-leipzig.de}

\author[Scheuermann]{Gerik Scheuermann}
\address{%
Universit\"at Leipzig\\
Institut f\"ur Informatik\\
Johannisgasse 26\\
04103 Leipzig, Germany}
\email{scheuermann@informatik.uni-leipzig.de}

\author[Hitzer]{Eckhard Hitzer}
\address{%
University of Fukui\\
Department of Applied Physics\\
3-9-1 Bunkyo\\
Fukui 910-8507, Japan}
\email{hitzer@mech.u-fukui.ac.jp}



\date{\today}

\begin{abstract}
 The increasing demand for Fourier transforms on geometric algebras has resulted in a large variety. Here we introduce one single straight forward definition of a general geometric Fourier transform covering most versions in the literature. We show which constraints are additionally necessary to obtain certain features like linearity or a shift theorem. As a result, we provide guidelines for the target-oriented design of yet unconsidered transforms that fulfill requirements in a specific application context. Furthermore, the standard theorems do not need to be shown in a slightly different form every time a new geometric Fourier transform is developed since they are proved here once and for all. 
\end{abstract}

%%% ----------------------------------------------------------------------
\maketitle
%%% ----------------------------------------------------------------------
%\tableofcontents
\section{Introduction}
The Fourier transform by Jean Baptiste Joseph Fourier is an indispensable tool for many fields of mathematics, physics, computer science and engineering. Especially the analysis and solution of differential equations or signal and image processing can not be imagined without it any more. 
Its kernel consists of the complex exponential function. With the square root of minus one, the imaginary unit $i$, as part of the argument it is periodic and therefore suitable for the analysis of oscillating systems.
\par
William Kingdon Clifford created the geometric algebras in 1878, \cite{C1878}. They usually contain continuous submanifolds of geometric square roots of minus one \cite{HA10,HHA11}. Each multivector has a natural geometric interpretation so the generalization of the Fourier transform to multivector valued functions in the geometric algebras is very reasonable. It helps to interpret the transform, apply it in a target oriented way to the specific underlying problem and allows a new point of view on fluid mechanics.
\par
Application oriented many different definitions of Fourier transforms in geometric algebras were developed. For example the Clifford Fourier transform introduced by Jancewicz \cite{Janc90} and expanded by Ebling and Scheuermann \cite{Ebl06} and Hitzer and Mawardi \cite{HM08} or the one established by Sommen in \cite{Som82} and re-established by B\"ulow \cite{Bue99}. Further we have the quaternionic Fourier transform by Ell \cite{Ell93} and later by B\"ulow \cite{Bue99}, the spacetime Fourier transform by Hitzer \cite{Hitz07}, the Clifford Fourier transform for color images by Batard et al. \cite{BBS08}, the Cylindrical Fourier transform by Brackx et al. \cite{BSS10}, the transforms by Felsberg \cite{Fels02} or Ell and Sangwine \cite{ES00,ES07}. All these transforms have different interesting properties and deserve to be studied independently from one another. But the analysis of their similarities reveals a lot about their qualities, too. We concentrate on this matter and summarize all of them in one general definition.
\par
Recently there have been very successful approaches by De Bie, Brackx, De Schepper and Sommen to construct Clifford Fourier transforms from operator exponentials and differential equations \cite{BSS05,BS06,BS09,BSS10a}. The definition presented in this paper does not cover all of them, partly because their closed integral form is not always known or highly complicated, and partly because they can be produced by combinations and functions of our transforms. 
\par
We focus on continuous geometric Fourier transforms over flat spaces $\R^{p,q}$ in their integral representation. That way their finite, regular discrete equivalents as used in computational signal and image processing can be intuitively constructed and direct applicability to the existing practical issues and easy numerical manageability are ensured. 
%
%---------------------------------------------------------------------------------------------------------------------------------
\section{Definition of the GFT}
%---------------------------------------------------------------------------------------------------------------------------------
We examine geometric algebras $\clifford{p,q},p+q=n\in\N$ over $\R^{p+q}$ \cite{HS84} generated by the associative, bilinear geometric product \index{geometric product} with neutral element $1$ satisfying 
\begin{eqnarray}\begin{aligned}
 \e_j\e_k+\e_k\e_j=\epsilon_j\delta_{jk},
\end{aligned}\end{eqnarray}
for all $j,k\in\{1,...,n\}$ with the Kronecker symbol $\delta$ and 
\begin{eqnarray}\begin{aligned}
\epsilon_j=\begin{cases}1&\forall j=1,...,p,\\-1&\forall j=p+1,...,n.\end{cases} 
\end{aligned}\end{eqnarray}
For the sake of brevity we want to refer to arbitrary multivectors 
\begin{eqnarray}\begin{aligned}
 \A=\sum\limits_{k=0}^n\sum\limits_{1\leq j_1<...<j_k\leq n}a_{j_1...j_k}\e_{j_1}...\e_{j_k}\in\clifford{p,q},
\end{aligned}\end{eqnarray}
$ a_{j_1...j_k} \in\R,$ as 
\begin{eqnarray}\label{mimv}
 \A=\sum_{\j}a_{\j}\e_{\j}.
\end{eqnarray}
where each of the $2^n$ multi-indices $\vec{j}\subseteq\{1,...,n\}$ indicates a basis vector of $\clifford{p,q}$ by $\e_{\j}=\e_{j_1}...\e_{j_k},$ $1\leq j_1<...<j_k\leq n,\e_{\emptyset}=\e_0=1$ and its associated coefficient $a_{\j}=a_{j_1...j_k}\in\R.$
%
%-----------------------------------------------------------------------------------------------------------------------------------
%
\begin{defn}\label{d:exp}
 The \textbf{exponential function} \index{exponential function} of a multivector $\A\in\clifford{p,q}$ is defined by the power series
\begin{eqnarray}\begin{aligned}
 e^{\A}&:=&\sum_{j=0}^{\infty}\frac{\A^j}{j!}.
\end{aligned}\end{eqnarray}
\end{defn}
\begin{lem}\label{l:potenzgesetz}
For two multivectors $\A\B=\B\A$ that commute amongst each other we have
 \begin{eqnarray}\begin{aligned}
 e^{\A+\B}=&e^{\A}e^{\B}.
\end{aligned}\end{eqnarray}
\end{lem}
\begin{proof}
 Analogous to the exponent rule of real matrices.
\end{proof}
\begin{bez}\label{b:im}
 For each geometric algebra $\clifford{p,q}$ we will write $\mathscr I^{p,q}=\{i\in\clifford{p,q},i^2\in\R^-\}$ to denote the real multiples of all geometric square roots of minus one \index{square roots of minus one}, compare \cite{HA10} and \cite{HHA11}. We chose the symbol $\mathscr I$ to be reminiscent of the imaginary numbers.
\end{bez}
\begin{defn}\label{d:gft}\index{geometric Fourier transform!definition}
 Let $\clifford{p,q}$ be a geometric Algebra, $\A:\R^m\to\clifford{p,q}$ be a multivector field and $\x,\u\in\R^m$ vectors. A \textbf{Geometric Fourier Transform} (GFT) \index{GFT} $\F_{F_1,F_2}(\A)$ is defined by
two ordered finite sets $F_1=\{f_1(\x,\u),...,f_{\mu}(\x,\u)\},$ $F_2=\{f_{\mu+1}(\x,\u),...,f_{\nu}(\x,\u)\}$ of mappings $f_k(\x,\u):\R^m\times \R^m\to\mathscr I^{p,q},\forall k=1,...,\nu$ and the calculation rule
\begin{eqnarray}\begin{aligned}
 \F_{F_1,F_2}(\A)(\u):=\int_{\R^m}\prod_{f\in F_1}e^{-f(\x,\u)}\A(\x)\prod_{f\in F_2}e^{-f(\x,\u)}\d^m \x.
\end{aligned}\end{eqnarray}
% operating from the left and analogously by
% \begin{eqnarray}\begin{aligned}
% \F_r(\A)(\u)=\int_{\R^2}\A(\x)e^{-2\pi \u\x}\d^2 \x
% \end{aligned}\end{eqnarray}
% operating from the right.
\end{defn}
This definition combines many Fourier transforms to a single general one. It enables us to proof the well known theorems just dependent on the properties of the chosen mappings.
\begin{ex}\label{bspgft}
Depending on the choice of $F_1$ and $F_2$ we get already developed transforms.
\begin{enumerate}
\item In the case of $\A:\R^n\to\G^{n,0}, n=2\pmod 4$ or $n=3\pmod 4$, we can reproduce the Clifford Fourier transform \index{Clifford Fourier transform}  introduced by Jancewicz \cite{Janc90} for $n=3$ and expanded by Ebling and Scheuermann \cite{Ebl06} for $n=2$ and Hitzer and Mawardi \cite{HM08} for $n=2\pmod 4$ or $n=3\pmod 4$ using the configuration
\begin{eqnarray}\begin{aligned}
F_1=&\emptyset,\\F_2=&\{f_1\},\\f_1(\x,\u)=&2\pi i_n \x\cdot\u,
\end{aligned}\end{eqnarray}
with $i_n$ being the pseudoscalar of $G^{n,0}$.
\item Choosing multivector fields $\R^n\to\G^{0,n}$,\index{Sommen Fourier transform} 
\begin{eqnarray}\begin{aligned}
F_1=&\emptyset,\\F_2=&\{f_1,...,f_n\},\\f_k(\x,\u)=&2\pi \e_kx_ku_k,\forall k=1,...,n
\end{aligned}\end{eqnarray}
we have the Sommen B\"ulow Clifford Fourier transform from \cite{Som82,Bue99}.
\item For $\A:\R^2\to\G^{0,2}\approx\mathbb{H}$ the quaternionic Fourier transform \index{quaternionic Fourier transform} \cite{Ell93,Bue99} is generated by 
\begin{eqnarray}\begin{aligned}
F_1=&\{f_1\},\\F_2=&\{f_2\},\\f_1(\x,\u)=&2\pi ix_1u_1,\\f_2(\x,\u)=&2\pi jx_2u_2.
\end{aligned}\end{eqnarray}
\item Using $\G^{3,1}$ we can build the spacetime respectively the volume-time Fourier transform \index{spacetime Fourier transform} from \cite{Hitz07}\footnote{Please note that Hitzer uses a different notation in \cite{Hitz07}. His $\x=t\e_0+x_1\e_1+x_2\e_2+x_3\e_3$ corresponds to our $\x=x_1\e_1+x_2\e_2+x_3\e_3+x_4\e_4$, with $\e_0\e_0=\epsilon_0=-1$ being equivalent to our $\e_4\e_4=\epsilon_4=-1$.} with the $\G^{3,1}$-pseudoscalar $i_4$ as follows
\begin{eqnarray}\begin{aligned}
F_1=&\{f_1\},\\F_2=&\{f_2\},\\f_1(\x,\u)=& \e_4x_4u_4,\\f_2(\x,\u)=& \epsilon_4\e_4i_4(x_1u_1+x_2u_2+x_3u_3).
\end{aligned}\end{eqnarray}
\item The Clifford Fourier transform for color images \index{color image Fourier transform} by Batard, Berthier and Saint-Jean \cite{BBS08} for $m=2,n=4,\A:\R^2\to\G^{4,0}$, a fixed bivector $\B$, and the pseudoscalar $i$ can intuitively be written as 
\begin{eqnarray}\begin{aligned}
F_1=&\{f_1\},\\F_2=&\{f_2\},\\f_1(\x,\u)=&\frac12 (x_1u_1+x_2u_2)(\B+i\B),\\f_2(\x,\u)=&-\frac12 (x_1u_1+x_2u_2)(\B+i\B),\end{aligned}\end{eqnarray}
but $(\B+i\B)$ does not square to a negative real number, see \cite{HA10}. The special property that $\B$ and $i\B$ commute amongst each other allows us to express the formula using
\begin{eqnarray}\begin{aligned}
F_1=&\{f_1,f_2\},\\F_2=&\{f_3,f_4\},\\f_1(\x,\u)=&\frac12 (x_1u_1+x_2u_2)\B,\\f_2(\x,\u)=&\frac12 (x_1u_1+x_2u_2)i\B,\\f_3(\x,\u)=&-\frac12 (x_1u_1+x_2u_2)\B,\\f_4(\x,\u)=&-\frac12 (x_1u_1+x_2u_2)i\B,
\end{aligned}\end{eqnarray}
which fulfills the conditions of Definition \ref{d:gft}.
%
%\item With the same choice of unit elements and according matrices but multivector fields in the hypercomplex algebras $HCA_d$ we generate the Commutative Hypercomplex Fourier transform.
%
\item Using $\G^{0,n}$ and
\begin{eqnarray}\begin{aligned}
F_1=&\{f_1\},\\F_2=&\emptyset,\\f_1(\x,\u)=&{-\x\wedge\u}
\end{aligned}\end{eqnarray}
produces the cylindrical Fourier transform \index{cylindrical Fourier transform} as introduced by Brackx, de Schepper and Sommen in \cite{BSS10}.
\end{enumerate}
\end{ex}
%
%---------------------------------------------------------------------------------------------------------------------------------
\section{General Properties}
%---------------------------------------------------------------------------------------------------------------------------------
%
First we proof general properties valid for arbitrary sets $F_1,F_2$.
\begin{thm}[Existence]\label{t:ex}\index{geometric Fourier transform!existence}
 The geometric Fourier transform exists for all integrable multivector fields $\A\in L_1(\R^n)$.
\end{thm}
\begin{proof}
The property 
\begin{eqnarray}\begin{aligned}
 f^2_k(\x,\u)\in\R^-
\end{aligned}\end{eqnarray}
of the mappings $f_k$ for $k=1,...,\nu$ leads to 
\begin{eqnarray}\begin{aligned}
\frac{f_k^2(\x,\u)}{|f_k^2(\x,\u)|}=-1
\end{aligned}\end{eqnarray}
for all $f_k(\x,\u)\neq 0$. So using the decomposition
\begin{eqnarray}\begin{aligned}
 f_k(\x,\u)=\frac{f_k(\x,\u)}{|f_k(\x,\u)|}|f_k(\x,\u)|
\end{aligned}\end{eqnarray}
we can write $\forall j\in\N$
\begin{eqnarray}\begin{aligned}
 f_k^j(\x,\u)=\begin{cases}(-1)^l|f_k(\x,\u)|^j&\text{for }j=2l, l\in\N_0% \text{ even,}
\\(-1)^l\frac{f_k(\x,\u)}{|f_k(\x,\u)|}|f_k(\x,\u)|^j&\text{for }j=2l+1 ,l\in\N_0%\text{ odd,}
\end{cases}
\end{aligned}\end{eqnarray}
which results in
\begin{eqnarray}\begin{aligned}\label{cossin}
 e^{-f_k(\x,\u)}=&\sum_{j=0}^{\infty}\frac{\big(-f_k(\x,\u)\big)^j}{j!}\\
=&\sum_{j=0}^{\infty}\frac{(-1)^j|f_k(\x,\u)|^{2j}}{(2j)!}
\\&-\frac{f_k(\x,\u)}{|f_k(\x,\u)|}\sum_{j=0}^{\infty}\frac{(-1)^j|f_k(\x,\u)|^{2j+1}}{(2j+1)!}\\
=&\cos\big(|f_k(\x,\u)|\big)-\frac{f_k(\x,\u)}{|f_k(\x,\u)|}\sin\big(|f_k(\x,\u)|\big)
\end{aligned}\end{eqnarray}
Because of 
\begin{eqnarray}\begin{aligned}
 |e^{-f_k(\x,\u)}|
=&\bigg|\cos\big(|f_k(\x,\u)|\big)-\frac{f_k(\x,\u)}{|f_k(\x,\u)|}\sin\big(|f_k(\x,\u)|\big)\bigg|\\
\leq&\bigg|\cos\big(|f_k(\x,\u)|\big)\bigg|+\bigg|\frac{f_k(\x,\u)}{|f_k(\x,\u)|}\bigg|\bigg|\sin\big(|f_k(\x,\u)|\big)\bigg|\\
\leq&2
\end{aligned}\end{eqnarray}
the magnitude of the improper integral 
\begin{eqnarray}\begin{aligned}
 |\F_{F_1,F_2}(\A)(\u)|=&\big|\int_{\R^m}\prod_{f\in F_1}e^{-f(\x,\u)}\A(\x)\prod_{f\in F_2}e^{-f(\x,\u)}\d^m \x\big|\\
\leq&\int_{\R^m}\prod_{f\in F_1}|e^{-f(\x,\u)}||\A(\x)|\prod_{f\in F_2}|e^{-f(\x,\u)}|\d^m \x\\
\leq&\int_{\R^m}\prod_{f\in F_1}2|\A(\x)|\prod_{f\in F_2}2\d^m \x\\
=&2^{\nu}\int_{\R^m}|\A(\x)|\d^m \x
\end{aligned}\end{eqnarray}
is finite and therefore the geometric Fourier transform exists.
\end{proof}
%
%---------------------------------------------------------------------------------------------------------------------------------
%
\begin{thm}[Scalar linearity]\label{t:lin}\index{geometric Fourier transform!linearity}
The geometric Fourier transform is linear with respect to scalar factors. Let $b,c\in\R$ and $\A,\B,\vec{C}:\R^m\to\clifford{p,q}$ be three multivector fields that satisfy $\A(\x)=b\B(\x)+c\vec C(\x)$, then
\begin{eqnarray}\begin{aligned}
 \F_{F_1,F_2}(\A)(\u)=&b\F_{F_1,F_2}(\B)(\u)+c\F_{F_1,F_2}(\vec C)(\u).
\end{aligned}\end{eqnarray}
\end{thm}
\begin{proof}
The assertion is an easy consequence of the distributivity of the geometric product over addition, the commutativity of scalars and the linearity of the integral.
\end{proof}
%
%
%---------------------------------------------------------------------------------------------------------------------------------
\section{Bilinearity}
%---------------------------------------------------------------------------------------------------------------------------------
%
All geometric Fourier transforms from the introductory example %\ref{bspgft} 
can also be expressed in terms of a stronger claim. The mappings $f_1,...,f_{\nu}$, with the first $\mu$ ones left of the argument function and the $\nu-\mu$ others on the right of it, are all bilinear and therefore take the form
\begin{eqnarray}
\begin{aligned}\label{bixu}
f_k(\x,\u)&=f_k(\sum_{j=1}^mx_j\e_j,\sum_{l=1}^mu_l\e_l)\\
&=\sum_{j,l=1}^mx_jf_k(\e_j,\e_l)u_l\\
&=\x^TM_k\u,
\end{aligned}
\end{eqnarray}
$\forall k=1,...,\nu$, where $M_k\in (\mathscr I^{p,q})^{m\times m},(M_k)_{jl}=f_k(\e_j,\e_l)$ according to Notation \ref{b:im}.
\begin{enumerate}
 \item In the Clifford Fourier transform $f_1$ can be written with 
\begin{eqnarray}\begin{aligned}
M_1=&2\pi i_n\Id.
\end{aligned}\end{eqnarray}
\item The $\nu=m=n$ mappings $f_k, k=1,...,n$ of the B\"ulow Clifford Fourier transform can be expressed using 
\begin{eqnarray}\begin{aligned}
(M_k)_{lj}=\begin{cases}2\pi\e_k&\text{ for }k=l=j,\\0&\text{ else.}\end{cases}
\end{aligned}\end{eqnarray}
\item Similarly the quaternionic Fourier transform is generated using 
\begin{eqnarray}\begin{aligned}
(M_1)_{l\iota}=\begin{cases}2\pi i&\text{ for }l=\iota=1,\\0&\text{ else,}\end{cases}\\
(M_2)_{l\iota}=\begin{cases}2\pi j&\text{ for }l=\iota=2,\\0&\text{ else.}\end{cases}
\end{aligned}\end{eqnarray}
% Please note that the imaginary unit and the index $j$ are not to be mixed up.
%
\item We can build the spacetime Fourier transform with 
\begin{eqnarray}\begin{aligned}
(M_1)_{lj}=&\begin{cases}\e_4&\text{ for }l=j=1,\\0&\text{ else,}\end{cases}\\
(M_2)_{lj}=&\begin{cases}\epsilon_4\e_4i_4&\text{ for }l=j\in\{2,3,4\},\\0&\text{ else.}\end{cases}
\end{aligned}\end{eqnarray}
\item The Clifford Fourier transform for color images can be described by 
\begin{eqnarray}\begin{aligned}
M_1=&\frac12\B\Id,\\
M_2=&\frac12i\B\Id,\\
M_3=&-\frac12\B\Id,\\
M_4=&-\frac12i\B\Id.
\end{aligned}\end{eqnarray}
 \item The cylindrical Fourier transform can also be reproduced with mappings satisfying (\ref{bixu}) because we can write
 \begin{eqnarray}\begin{aligned}\label{wedge}
 \x\wedge \u=&\e_1\e_2x_1u_2-\e_1\e_2x_2u_1
%\\&+\e_1\e_3x_1u_3-\e_1\e_3x_3u_1
\\&+...
\\&+\e_{m-1}\e_mx_{m-1}u_{m}-\e_{m-1}\e_mx_{m}u_{m-1}
\end{aligned}\end{eqnarray}
and set 
 \begin{eqnarray}\begin{aligned}
(M_1)_{lj}=\begin{cases}0&\text{ for }l=j,\\\e_l\e_j&\text{ else.}\end{cases}
 \end{aligned}\end{eqnarray}
\end{enumerate}
\begin{thm}[Scaling]\label{t:scaling}\index{geometric Fourier transform!scaling theorem}
 Let $0\neq a\in\R$ be a real number, $\A(\x)=\B(a\x)$ two multivector fields and all $F_1,F_2$ be bilinear mappings then the geometric Fourier transform satisfies
\begin{eqnarray}\begin{aligned}
\F_{F_1,F_2}(\A)(\u)=&|a|^{-m}\F_{F_1,F_2}(\B)\big(\frac{\u}a\big).
\end{aligned}\end{eqnarray}
\end{thm}
\begin{proof}
A change of coordinates together with the bilinearity proves the assertion by
\begin{eqnarray}\begin{aligned}
 \F_{F_1,F_2}(\A)(\u)%=&\int_{\R^m}\prod_{f\in F}e^{-f(\x,\u)}\A(\x)\prod_{b\in B}e^{-b(\x,\u)}\d^m \x\\
=&\int_{\R^m}\prod_{f\in F}e^{-f(\x,\u)}\B(a\x)\prod_{f\in B}e^{-f(\x,\u)}\d^m \x\\
\overset{a\x=\y}=&\int_{\R^m}\prod_{f\in F}e^{-f(\frac{\y}a,\u)}\B(\y)\prod_{f\in B}e^{-f(\frac{\y}a,\u)}|a|^{-m}\d^m \y\\
\overset{f\text{ bilin.}}=&|a|^{-m}\int_{\R^m}\prod_{f\in F}e^{-f(\y,\frac{\u}a)}\B(\y)\prod_{f\in B}e^{-f(\y,\frac{\u}a)}\d^m \y\\
=&|a|^{-m}\F_{F_1,F_2}(\B)\big(\frac{\u}a\big).
\end{aligned}\end{eqnarray}
\end{proof}
%
%---------------------------------------------------------------------------------------------------------------------------------
\section{Products with Invertible Factors}
%---------------------------------------------------------------------------------------------------------------------------------
%
To obtain properties of the GFT like linearity with respect to arbitrary multivectors or a shift theorem we will have to change the order of multivectors and products of exponentials. Since the geometric product usually is neither commutative nor anticommutative this is not trivial. In this section we provide useful Lemmata that allow a swap if at least one of the factors is invertible. For more information see \cite{HS84} and \cite{HHA11}.
\begin{rem}
 Every multiple of a square root of minus one $i\in\mathscr I^{p,q}$ is invertible, since from $i^2=-r,r\in\R\setminus\{0\}$ follows $i^{-1}=-\frac ir$. Because of that for all $\u,\x\in\R^m$ a function $f_k(\x,\u):\R^m\times \R^m\to\mathscr I^{p,q}$ is pointwise invertible.% (not in the sense of bijectivity).
\end{rem}
\begin{defn}\label{d:c}
For an invertible multivector $\B\in\clifford{p,q}$ and an arbitrary multivector $\A\in\clifford{p,q}$ we define \index{commutative part} \index{anticommutative part} \index{split w.r.t. commutativity}
\begin{eqnarray}\begin{aligned}
 \A_{\vec c^0(\B)}=&\frac12(\A+\B^{-1}\A\B),\\
\A_{\vec c^1(\B)}=&\frac12(\A-\B^{-1}\A\B).
\end{aligned}\end{eqnarray}
\end{defn}
\begin{lem}\label{l:comanticom}
Let $\B\in\clifford{p,q}$ be invertible with the unique inverse $\B^{-1}=\frac{\bar{\B}}{\B^2},\B^2\in\R\setminus\{0\}$. Every multivector $\A\in\clifford{p,q}$ can be expressed unambiguously by the sum of $\A_{\vec c^0(\B)}\in\clifford{p,q}$ that commutes and $\A_{\vec c^1(\B)}\in\clifford{p,q}$ that anticommutes with respect to $\B$. That means
\begin{eqnarray}\begin{aligned}\label{komakom}
 \A&=\A_{\vec c^0(\B)}+\A_{\vec c^1(\B)},\\
\A_{\vec c^0(\B)}\B&=\B\A_{\vec c^0(\B)},\\
\A_{\vec c^1(\B)}\B&=-\B\A_{\vec c^1(\B)}.
\end{aligned}\end{eqnarray}
\end{lem}
\begin{proof} We will only prove the assertion for $\A_{\vec c^0(\B)}$.\\
\textbf{Existence:}
With Definition \ref{d:c} we get
\begin{eqnarray}\begin{aligned}
 \A_{\vec c^0(\B)}+\A_{\vec c^1(\B)}=&\frac12(\A+\B^{-1}\A\B+\A-\B^{-1}\A\B)\\
=&\A
\end{aligned}\end{eqnarray}
and considering
\begin{eqnarray}\begin{aligned}
 \B^{-1}\A\B=\frac{\bar{\B}\A\B}{\B^2}=\B\A\B^{-1}
\end{aligned}\end{eqnarray}
we also get
\begin{eqnarray}\begin{aligned}
 \A_{\vec c^0(\B)}\B=&\frac12(\A+\B^{-1}\A\B)\B\\
=&\frac12(\A+\B\A\B^{-1})\B\\
=&\frac12(\A\B+\B\A)\\
=&\B\frac12(\B^{-1}\A\B+\A)\\
=&\B\A_{\vec c^0(\B)}
\end{aligned}\end{eqnarray}
\textbf{Uniqueness:}
From the first claim in (\ref{komakom}) we get
 \begin{eqnarray}\begin{aligned}
\A_{\vec c^1(\B)}=&\A-\A_{\vec c^0(\B)},
\end{aligned}\end{eqnarray}
together with the third one this leads to
\begin{eqnarray}\begin{aligned}
(\A-\A_{\vec c^0(\B)})\B=&-\B(\A-\A_{\vec c^0(\B)})\\
\A\B-\A_{\vec c^0(\B)}\B=&-\B\A+\B\A_{\vec c^0(\B)}\\
\A\B+\B\A=&\A_{\vec c^0(\B)}\B+\B\A_{\vec c^0(\B)}
\end{aligned}\end{eqnarray}
and from the second claim finally follows
\begin{eqnarray}\begin{aligned}
\A\B+\B\A=&2\B\A_{\vec c^0(\B)}\\
\frac12(\B^{-1}\A\B+\A)=&\A_{\vec c^0(\B)}.
\end{aligned}\end{eqnarray}
The derivation of the expression for $\A_{\vec c^1(\B)}$ works analogously.
\end{proof}
\begin{cor}[Decomposition w.r.t. commutativity]\label{k:prodeins}
 Let $\B\in\clifford{p,q}$ be invertible, then $\forall\A\in\clifford{p,q}$
\begin{eqnarray}\begin{aligned}
 \B\A=(\A_{\vec c^0(\B)}-\A_{\vec c^1(\B)})\B.
\end{aligned}\end{eqnarray}
\end{cor}
%
%----------------------------------------------------------------------------------------------------------------------------------
%
\begin{defn}
 For $d\in\N,\A\in\clifford{p,q}$, the ordered set $B=\{\B_1,...,\B_d\}$ of invertible multivectors and any multi-index $\j\in\{0,1\}^{d}$ we define
\begin{eqnarray}\begin{aligned}
 \A_{\vec c^{\j}(\overrightarrow B)}:=&((\A_{\vec c^{j_1}(\B_1)})_{\vec c^{j_2}(\B_2)}...)_{\vec c^{j_d}(\B_d)},\\
 \A_{\vec c^{\j}(\overleftarrow B)}:=&((\A_{\vec c^{j_d}(\B_d)})_{\vec c^{j_{d-1}}(\B_{d-1})}...)_{\vec c^{j_1}(\B_1)}
\end{aligned}\end{eqnarray}
recursively with $\vec c^0, \vec c^1$ of Definition \ref{d:c}.
\end{defn}
\begin{ex}
 Let $\A=a_0+a_1\e_1+a_2\e_2+a_{12}\e_{12}\in\G^{2,0}$ then for example
\begin{eqnarray}\begin{aligned}
 \A_{\vec c^{0}(\e_{1})}=&\frac 12(\A+\e_1^{-1}\A\e_1)
\\=&\frac 12(\A +a_0+a_1\e_1-a_2\e_2-a_{12}\e_{12})
\\=&a_0+a_1\e_1
\end{aligned}\end{eqnarray}
and further
\begin{eqnarray}\begin{aligned}
 \A_{\vec c^{0,0}(\overrightarrow {\e_{1},\e_{2}})}=&(\A_{\vec c^{0}(\e_{1})})_{\vec c^{0}(\e_{2})}
\\=&(a_0+a_1\e_1)_{\vec c^{0}(\e_{2})}=a_0.
\end{aligned}\end{eqnarray}
% \begin{eqnarray}\begin{aligned}
% \vec c^{0,0,0 }(\e_{1},\e_{2},\e_{12})=\vec c^{0}(\e_{1})\cap\vec c^{0}(\e_{2})\cap\vec c^{0}(\e_{12})=\{0,1\}\cap\{0,2\}\cap\{0,12\}=\{0\}
% \end{aligned}\end{eqnarray}
The computation of the other multi-indices with $d=2$ works analogously and therefore
\begin{eqnarray}\begin{aligned}
 \A=&\sum_{\j\in\{0,1\}^d}\A_{\vec{c}^{\j}(\e_{1},\e_{2})}\\
=&\A_{\vec{c}^{00}(\overrightarrow{\e_{1},\e_{2}})}+\A_{\vec{c}^{01}(\overrightarrow{\e_{1},\e_{2}})}+\A_{\vec{c}^{10}(\overrightarrow{\e_{1},\e_{2}})}+\A_{\vec{c}^{11}(\overrightarrow{\e_{1},\e_{2}})}\\
=&a_0+a_1\e_1+a_2\e_2+a_{12}\e_{12}.
\end{aligned}\end{eqnarray}
\end{ex}
\begin{lem}\label{l:prodviele}
 Let $d\in\N,B=\{\B_1,...,\B_d\}$ be invertible multivectors and for $\j\in\{0,1\}^{d}$ let $|\j|:=\sum_{k=1}^dj_k$, then $\forall\A\in\clifford{p,q}$
\begin{eqnarray}\begin{aligned}
\A=&\sum_{\j\in\{0,1\}^{d}}\A_{\vec c^{\j}(\overrightarrow B)},\\
\A\B_1...\B_d=&\B_1...\B_d\sum_{\j\in\{0,1\}^{d}}(-1)^{|\j|}\A_{\vec c^{\j}(\overrightarrow B)},\\
\B_1...\B_d\A=&\sum_{\j\in\{0,1\}^{d}}(-1)^{|\j|}\A_{\vec c^{\j}(\overleftarrow B)}\B_1...\B_d.
\end{aligned}\end{eqnarray}
\end{lem}
\begin{proof}
 Apply Lemma \ref{l:comanticom} repeatedly.
\end{proof}
\begin{rem}
 The distinction of the two directions can be omitted using the equality
\begin{eqnarray}\begin{aligned}
 \A_{\vec c^{\j}(\overrightarrow {\B_1,...,\B_d})}= \A_{\vec c^{\j}(\overleftarrow {\B_d,...,\B_1})}.
\end{aligned}\end{eqnarray}
We established it for the sake of notational brevity and will not formulate nor proof every assertion for both directions.
\end{rem}
%
%----------------------------------------------------------------------------------------------------------------------------------
%
\begin{lem}\label{l:expcom}
 Let $F=\{f_1(\x,\u),...,f_{d}(\x,\u)\}$ be a set of pointwise invertible functions then the ordered product of their exponentials and an arbitrary multivector $\A\in\clifford{p,q}$ satisfies
\begin{eqnarray}\begin{aligned}
 \prod_{k=1}^{d}e^{-f_k(\x,\u)}\A=\sum_{\j\in\{0,1\}^d}\vec A_{\vec c^{\j}(\overleftarrow F)}(\x,\u)\prod_{k=1}^{d}e^{-(-1)^{j_k}f_k(\x,\u)},
\end{aligned}\end{eqnarray}
where $A_{\vec c^{\j}(\overleftarrow F)}(\x,\u):=A_{\vec c^{\j}(\overleftarrow {F(\x,\u)})}$ is a multivector valued function $\R^m\times\R^m\to\clifford{p,q}$.
\end{lem}
\begin{proof}
 For all $\x,\u\in\R^m$ the commutation properties of $f_k(\x,\u)$ dictate the ones of $e^{-f_k(\x,\u)}$ by
\begin{eqnarray}\begin{aligned}
 e^{-f_k(\x,\u)}\A\overset{\text{Def. }\ref{d:exp}}=&\sum_{l=0}^{\infty}\frac{(-f_k(\x,\u))^l}{l!}\A\\
\overset{\text{Lem. }\ref{l:comanticom}}=&\sum_{l=0}^{\infty}\frac{(-f_k(\x,\u))^l}{l!}(\A_{\vec c^0(f_k(\x,\u))}+\A_{\vec c^1(f_k(\x,\u))}).
\end{aligned}\end{eqnarray}
The shape of this decomposition of $\A$ may depend on $\x$ and $\u$. To stress this fact we will interpret $\A_{\vec c^0(f_k(\x,\u))}$ as a multivector function and write $\A_{\vec c^0(f_k)}(\x,\u)$. According to Lemma \ref{l:comanticom} we can move $\A_{\vec c^0(f_k)}(\x,\u)$ through all factors, because it commutes. Analogously swapping $\A_{\vec c^1(f_k)}(\x,\u)$ will change the sign of each factor because it anticommutes. Hence we get
\begin{eqnarray}\begin{aligned}
=&\A_{\vec c^0(f_k)}(\x,\u)\sum_{l=0}^{\infty}\frac{(-f_k(\x,\u))^l}{l!}+\A_{\vec c^1(f_k)}(\x,\u)\sum_{l=0}^{\infty}\frac{(f_k(\x,\u))^l}{l!}\\
=&\A_{\vec c^0(f_k)}(\x,\u)e^{-f_k(\x,\u)}+\A_{\vec c^1(f_k)}(\x,\u)e^{f_k(\x,\u)}.
\end{aligned}\end{eqnarray}
% and for $f$ that are linear with respect to the second argument additionally by
% \begin{eqnarray}\begin{aligned}
% e^{-f(\x,\u)}\A=&\A_{\vec c^0(f)}e^{-f(\x,\u)}+\A_{\vec c^1(f)}e^{f(\x,\u)}\\
% =&\A_{\vec c^0(f)}e^{-f(\x,\u)}+\A_{\vec c^1(f)}e^{-f(\x,-\u)}.
% \end{aligned}\end{eqnarray}
Applying this repeatedly to the product we can deduce
\begin{eqnarray}\begin{aligned}
 \prod_{k=1}^{d}e^{-f_k(\x,\u)}\A=&\prod_{k=1}^{d-1}e^{-f_k(\x,\u)}(\A_{\vec c^0(f_d)}(\x,\u)e^{-f_d(\x,\u)}
\\&+\A_{\vec c^1(f_d)}(\x,\u)e^{f_d(\x,\u)})\\
=&\prod_{k=1}^{d-2}e^{-f_k(\x,\u)}(\A_{\vec c^{0,0}(\overleftarrow{f_{d-1},f_d})}(\x,\u)e^{-f_{d-1}(\x,\u)}e^{-f_d(\x,\u)}\\
&+\A_{\vec c^{1,0}(\overleftarrow{f_{d-1},f_d})}(\x,\u)e^{f_{d-1}(\x,\u)}e^{-f_d(\x,\u)})\\
&+\A_{\vec c^{0,1}(\overleftarrow{f_{d-1},f_d})}(\x,\u)e^{-f_{d-1}(\x,\u)}e^{f_d(\x,\u)}\\
&+\A_{\vec c^{1,1}(\overleftarrow{f_{d-1},f_d})}(\x,\u)e^{f_{d-1}(\x,\u)}e^{f_d(\x,\u)})\\
=&...\\
=&\sum_{\j\in\{0,1\}^d}\vec A_{\vec c^{\j}(\overleftarrow F)}(\x,\u)\prod_{k=1}^{d}e^{-(-1)^{j_k}f_k(\x,\u)}.
\end{aligned}\end{eqnarray}
\end{proof}
%
%---------------------------------------------------------------------------------------------------------------------------------
\section{Separable GFT}
%---------------------------------------------------------------------------------------------------------------------------------
%
From now on we want to restrict ourselves to an important group of geometric Fourier transforms whose square roots of -1 are independent from the first argument.
\begin{defn}\label{d:sep}
We call a GFT \textbf{left (right) separable}, \index{separability} if
\begin{equation}
\label{ifix}
f_l=|f_l(\x,\u)|i_l(\u),
\end{equation}
$\forall l=1,...,\mu$, ($l=\mu+1,...,\nu$), where $|f_l(\x,\u)|:\R^m\times\R^m\to\R$ is a real function and $i_l:\R^m\to\mathscr I^{p,q}$ a function that does not depend on $\x$.
\end{defn}
%
% Let us study an important group of geometric Fourier transforms whose mappings $F_1$ or $F_2$ of Definition \ref{d:gft} can be written as
% \begin{eqnarray}
% \label{ifix}
% f_k=i_k(\u)g_k(\x,\u),
% \end{eqnarray}
% $\forall k=1,...,\mu$, respectively $k=\mu+1,...,\nu$, where $g_{k}:\R^m\times\R^m\to\R$ is a real function and $i_k:\R^m\to\mathscr I^{p,q}$ a function that does not depend on $\x$. If a representation like this is possible, we can assume $i_k$ to be unit, i.e. $\forall \u\in\R^m:i_k^2(\u)=-1$.
% %
\begin{ex}
 The first five transforms from the introductory example %\ref{bspgft}
 are separable, while 
the cylindrical transform (vi) can not be expressed in the way of (\ref{ifix}) except for the two dimensional case.
\end{ex}
 We have seen in the proof of Lemma \ref{l:expcom} that the decomposition of a constant multivector $\vec A$ with respect to a product of exponentials generally results in multivector valued functions $\vec A_{\vec c^{\j}(F)}(\x,\u)$ of $\x$ and $\u$. Separability guarantees independence from $\x$ and therefore allows separation from the integral. 
\begin{cor}[Decomposition independent from $\x$]\label{k:expcom}
 Consider a set of functions $F=\{f_1(\x,\u),...,f_{d}(\x,\u)\}$ satisfying condition (\ref{ifix}) then the ordered product of their exponentials and an arbitrary multivector $\A\in\clifford{p,q}$ satisfies
\begin{eqnarray}\begin{aligned}
 \prod_{k=1}^{d}e^{-f_k(\x,\u)}\A=\sum_{\j\in\{0,1\}^d}\vec A_{\vec c^{\j}(\overleftarrow F)}(\u)\prod_{k=1}^{d}e^{-(-1)^{j_k}f_k(\x,\u)}.
\end{aligned}\end{eqnarray}
\end{cor}
\begin{rem}
 If a GFT can be expressed as in (\ref{ifix}) but with multiples of square roots of $-1$ $i_k\in\mathscr I^{p,q}$, which are independent from $\x$ and $\u$, the parts $\vec A_{\vec c^{\j}(\overleftarrow F)}$ of $\A$ will be constants. Note that the first five GFTs from the reference example %\ref{bspgft}
 satisfy this stronger condition, too.
\end{rem}
%
%---------------------------------------------------------------------------------------------------------------------------------
%\subsection{Products}
%---------------------------------------------------------------------------------------------------------------------------------
%
\begin{defn}\label{d:fj}
 For a set of functions $F=\{f_1(\x,\u),...,f_{d}(\x,\u)\}$ and a multi-index ${\j\in\{0,1\}^d}$, we define the set of functions $F(\j)$ by
\begin{eqnarray}\begin{aligned}
 F(\j):=\{(-1)^{j_1}f_1(\x,\u),...,(-1)^{j_{d}}f_{d}(\x,\u)\}.
\end{aligned}\end{eqnarray}
\end{defn}
\begin{thm}[Left and right products]\label{t:products}\index{geometric Fourier transform!product theorem}
Let $\vec{C}\in\clifford{p,q}$ and $\A,\B:\R^m\to\clifford{p,q}$ be two multivector fields with $\A(\x)=\vec{C}\B(\x)$ then a left separable geometric Fourier transform obeys
\begin{eqnarray}\begin{aligned}
 \F_{F_1,F_2}(\A)(\u)=&\sum_{\j\in\{0,1\}^{\mu}}\vec C_{\vec c^{\j}(\overleftarrow{F_1})}(\u)\F_{F_1(\j),F_2}(\B)(\u).
\end{aligned}\end{eqnarray}
If $\A(\x)=\B(\x)\vec{C}$ we analogously get
\begin{eqnarray}\begin{aligned}
 \F_{F_1,F_2}(\A)(\u)=&\sum_{\k\in\{0,1\}^{(\nu-\mu)}}\F_{F_1,F_2(\k)}(\B)(\u)\vec C_{\vec c^{\k}(\overrightarrow{F_2})}(\u)
\end{aligned}\end{eqnarray}
for a right separable GFT.
\end{thm}
\begin{proof}
We restrict ourselves to the proof of the first assertion.
\begin{eqnarray}\begin{aligned}
 \F_{F_1,F_2}(\A)(\u)=&\int_{\R^m}\prod_{f\in F_1}e^{-f(\x,\u)}\vec{C}\B(\x)\prod_{f\in F_2}e^{-f(\x,\u)}\d^m \x\\
\overset{\text{Lem. }\ref{l:expcom}}=&\int_{\R^m}(\sum_{\j\in\{0,1\}^{\mu}}\vec C_{\vec c^{\j}(\overleftarrow{F_1})}(\u)\prod_{l=1}^{\mu}e^{-(-1)^{j_l}f_l(\x,\u)})
\\&\B(\x)\prod_{f\in F_2}e^{-f(\x,\u)}\d^m \x\\
=&\sum_{\j\in\{0,1\}^{\mu}}\vec C_{\vec c^{\j}(\overleftarrow{F_1})}(\u)\int_{\R^m}\prod_{l=1}^{\mu}e^{-(-1)^{j_l}f_l(\x,\u)}
\\&\B(\x)\prod_{f\in F_2}e^{-f(\x,\u)}\d^m \x\\
=&\sum_{\j\in\{0,1\}^{\mu}}\vec C_{\vec c^{\j}(\overleftarrow{F_1})}(\u)\F_{F_1(\j),F_2}(\B)(\u)
\end{aligned}\end{eqnarray}
The second one follows in the same way.
% \begin{eqnarray}\begin{aligned}
% \F_{F_1,F_2}(\A)(\u)=&\int_{\R^m}\prod_{f\in F_1}e^{-f(\x,\u)}\vec{C}\B(\x)\vec D\prod_{f\in f_2}e^{-f(\x,\u)}\d^m \x\\
% =&\int_{\R^m}(\vec C_{\vec c(F_1)}-\vec C_{\bar {\vec c}(F_1)})\prod_{f\in F_1}e^{-f(\x,\u)}\B(\x)\prod_{f\in F_2}e^{-f(\x,\u)}(\vec D_{\vec c(F_2)}-\vec D_{\bar {\vec c}(F_2)})\d^m \x\\
% =&(\vec C_{\vec c(F_1)}-\vec C_{\bar {\vec c}(F_1)})\F_{F_1,F_2}(\B)(\u)+(\vec D_{\vec c(F_2)}-\vec D_{\bar {\vec c}(F_2)}).
% \end{aligned}\end{eqnarray}
\end{proof}
\begin{cor}[Uniform constants]\label{k:uniformC}
 Let the claims from Theorem \ref{t:products} hold. If the constant $\vec C$ satisfies 
$\vec C=\vec C_{\vec c^{\j}(\overleftarrow{F_1})}(\u)$ for a multi-index $\j\in\{0,1\}^{\mu}$ then the theorem simplifies to
\begin{eqnarray}\begin{aligned}
 \F_{F_1,F_2}(\A)(\u)=&\vec C\F_{F_1(\j),F_2}(\B)(\u)
\end{aligned}\end{eqnarray}
for $\A(\x)=\vec{C}\B(\x)$ respectively
\begin{eqnarray}\begin{aligned}
 \F_{F_1,F_2}(\A)(\u)=&\F_{F_1,F_2(\k)}(\B)(\u)\vec C
\end{aligned}\end{eqnarray}
for $\A(\x)=\B(\x)\vec{C}$ and $\vec C=\vec C_{\vec c^{\k}(\overrightarrow{F_2})}(\u)$ for a multi-index $\k\in\{0,1\}^{(\nu-\mu)}$.\footnote{Corollary \ref{k:uniformC} follows directly from $(\vec C_{\vec c^{\j}(\overleftarrow{F_1})})_{\vec c^{\k}(\overleftarrow{F_1})}=0$ for all $\k\neq\j$ %and $(\vec C_{\vec c^{\j}(\overleftarrow{F_1})})_{\vec c^{\j}(\overleftarrow{F_1})}=\vec C_{\vec c^{\j}(\overleftarrow{F_1})}$
because no non-zero component of $\vec C$ can commute and anticommute with respect to a function in $F_1$.}
\end{cor}
%
% \begin{proof}From the claim of the corollary and Lemma \ref{l:prodviele} follows
% \begin{eqnarray}\begin{aligned}
% \vec C=\sum_{\k\in\{0,1\}^{\mu}}\vec C_{\vec c^{\k}(\overleftarrow{F_1})}(\u)=\vec C_{\vec c^{\j}(\overleftarrow{F_1})}(\u)
% \Rightarrow\sum_{\k\in\{0,1\}^{\mu}\neq\j}\vec C_{\vec c^{\j}(\overleftarrow{F_1})}(\u)=0
% \end{aligned}\end{eqnarray}
% and therefore we get
% \begin{eqnarray}\begin{aligned}
% \F_{F_1,F_2}(\A)(\u)=&\F_{F_1,F_2}(\vec C\B)(\u)\\
% \overset{Lem. \ref{l:prodviele}}=&\F_{F_1,F_2}(\vec C_{\j(\overleftarrow{F_1})}(\u)+\sum_{\k\in\{0,1\}^{\mu}\neq\j}\vec C_{\vec c^{\k}(\overleftarrow{F_1})}(\u)\B)(\u)\\
% =&\vec C\F_{F_1(\j),F_2}(\B)(\u)
% \end{aligned}\end{eqnarray}
% \end{proof}
%
\begin{cor}[Left and right linearity]\label{k:lin}
The geometric Fourier transform is left (respectively right) linear if $F_1$ (respectively $F_2$) only consists of functions $f_k$ with values in the center of $\clifford{p,q}$, that means $\forall \x,\u\in\R^m,\forall \A\in\clifford{p,q}:\A f_k(\x,\u)=f_k(\x,\u) \A$. 
\end{cor}
\begin{rem}
 Note that for empty sets $F_{1}$ (or $F_2$) necessarily all elements satisfy commutativity and therefore the condition in corollary \ref{k:lin}.
\end{rem}
 The different appearances of Theorem \ref{t:products} are summarized in Table 1 and Table 2.
%
%\vspace{12pt}
\begin{table}[hbt!]
\begin{tabular}{|l|l|l|}
\hline
&GFT		& $\A(\x)=\vec{C}\B(\x)$	 	\rule [-1.8mm]{0mm}{6mm}\\\hline
1.& Clifford	& $\F_{f_1}=\vec C\F_{f_1}$ 	 \rule [-1.8mm]{0mm}{6mm}\\
2.& B\"ulow		& $\F_{f_1,...,f_n}=\vec C\F_{f_1,...,f_n}$ 			\rule [-1.8mm]{0mm}{6mm}\\
3.& Quaternionic	& $\F_{f_1,f_2}=\vec C_{\vec c^{0}(i)}\F_{f_1,f_2}+\vec C_{\vec c^{1}(i)}\F_{-f_1,f_2}$	 		 \rule [-1.8mm]{0mm}{6mm}\\
4.& Spacetime	& $\F_{f_1,f_2}=\vec C_{\vec c^{0}(\e_4)}\F_{f_1,f_2}+\vec C_{\vec c^{1}(\e_4)}\F_{-f_1,f_2}$			\rule [-1.8mm]{0mm}{6mm}\\
5.& Color Image	&$\F_{f_1,f_2,f_3,f_4}=\vec C_{\vec c^{00}(\overleftarrow{\B,i\B})}\F_{f_1,f_2,f_3,f_4}$\rule [-1.8mm]{0mm}{6mm}\\&&\phantom{$\F_{f_1,f_2,f_3,f_4}=$}$
+\vec C_{\vec c^{10}(\overleftarrow{\B,i\B})}\F_{-f_1,f_2,f_3,f_4}$\rule [-1.8mm]{0mm}{6mm}\\&&\phantom{$\F_{f_1,f_2,f_3,f_4}=$}$
+\vec C_{\vec c^{01}(\overleftarrow{\B,i\B})}\F_{f_1,-f_2,f_3,f_4}$\rule [-1.8mm]{0mm}{6mm}\\&&\phantom{$\F_{f_1,f_2,f_3,f_4}=$}$
+\vec C_{\vec c^{11}(\overleftarrow{\B,i\B})}\F_{-f_1,-f_2,f_3,f_4}$ 	\rule [-1.8mm]{0mm}{6mm}\\
6.& Cylindrical $n=2$	& $\F_{f_1}=\vec C_{\vec c^{0}(\e_{12})}\F_{f_1}+\vec C_{\vec c^{1}(\e_{12})}\F_{-f_1}$				\rule [-1.8mm]{0mm}{6mm}\\
& Cylindrical $n\neq2$	& -				\rule [-1.8mm]{0mm}{6mm}\\\hline
\end{tabular}%\label{tab:g2}
\caption{Theorem \ref{t:products} (Left products) applied to the GFTs of the first example %\ref{bspgft}
 enumerated in the same order. Notations: on the LHS $\F_{F_1,F_2}=\F_{F_1,F_2}(\A)(\u)$, on the RHS $\F_{F^{\prime}_1,F^{\prime}_2}=\F_{F^{\prime}_1,F^{\prime}_2}(\B)(\u)$}
\end{table}
%
%\vspace{12pt}
\begin{table}[hbt!]
\begin{tabular}{|l|l|l|}
\hline
&GFT		& $\A(\x)=\B(\x)\vec{C}$	 	\rule [-1.8mm]{0mm}{6mm}\\\hline
1.& Clif. $n=2\pmod4$& $\F_{f_1}=\F_{f_1}\vec C_{\vec c^{0}(i)}+\F_{-f_1}\vec C_{\vec c^{1}(i)}$ \rule [-1.8mm]{0mm}{6mm}\\
& Clif. $n=3\pmod4$& $\F_{f_1}=\F_{f_1}\vec C$ 	 \rule [-1.8mm]{0mm}{6mm}\\
2.& B\"ulow		& $\F_{f_1,...,f_n}$
\rule [-1.8mm]{0mm}{6mm}\\&&\phantom{$ $}$
=\sum_{\k\in\{0,1\}^{n}}\F_{(-1)^{k_1}f_1,...,(-1)^{k_n}f_n}\vec C_{\vec c^{\k}(\overrightarrow{f_1,...,f_n})}$ 			\rule [-1.8mm]{0mm}{6mm}\\
3.& Quaternionic	& $\F_{f_1,f_2}=\F_{f_1,f_2}\vec C_{\vec c^{0}(j)}+\F_{f_1,-f_2}\vec C_{\vec c^{1}(j)}$	 \rule [-1.8mm]{0mm}{6mm}\\
4.& Spacetime	& $\F_{f_1,f_2}=\F_{f_1,f_2}\vec C_{\vec c^{0}(\e_4i_4)}+\F_{f_1,-f_2}\vec C_{\vec c^{1}(\e_4i_4)}$	 \rule [-1.8mm]{0mm}{6mm}\\
5.& Color Image	&$\F_{f_1,f_2,f_3,f_4}=\F_{f_1,f_2,f_3,f_4}\vec C_{\vec c^{00}(\overrightarrow{\B,i\B})}$\rule [-1.8mm]{0mm}{6mm}\\&&\phantom{$\F_{f_1,f_2,f_3,f_4}+$}$
+\F_{f_1,f_2,-f_3,f_4}\vec C_{\vec c^{10}(\overrightarrow{\B,i\B})}$\rule [-1.8mm]{0mm}{6mm}\\&&\phantom{$\F_{f_1,f_2,f_3,f_4}+$}$
+\F_{f_1,f_2,f_3,-f_4}\vec C_{\vec c^{01}(\overrightarrow{\B,i\B})}$\rule [-1.8mm]{0mm}{6mm}\\&&\phantom{$\F_{f_1,f_2,f_3,f_4}+$}$
+\F_{f_1,f_2,-f_3,-f_4}\vec C_{\vec c^{11}(\overrightarrow{\B,i\B})}$ 	\rule [-1.8mm]{0mm}{6mm}\\
6.& Cylindrical	& $\F_{f_1}=\F_{f_1}\vec C$				\rule [-1.8mm]{0mm}{6mm}\\\hline
\end{tabular}%\label{tab:g2}
\caption{Theorem \ref{t:products} (Right products) applied to the GFTs of the first example%\ref{bspgft}
, enumerated in the same order. Notations: on the LHS $\F_{F_1,F_2}=\F_{F_1,F_2}(\A)(\u)$, on the RHS $\F_{F^{\prime}_1,F^{\prime}_2}=\F_{F^{\prime}_1,F^{\prime}_2}(\B)(\u)$}
\end{table}
%
%---------------------------------------------------------------------------------------------------------------------------------
%\subsection{Shift Theorem}
%---------------------------------------------------------------------------------------------------------------------------------
\par
We have seen how to change the order of a multivector and a product of exponentials in the previous section. To get a shift theorem we will have to separate sums appearing in the exponent and sort the resulting exponentials with respect to the summands. Note that corollary \ref{k:expcom} can be applied in two ways here, because exponentials appear on both sides. \par
Not every factor will need to be swapped with every other one. So, to keep things short, we will make use of the notation $\c^{(J)_l}(f_1,...,f_l,0,...,0)$ for $l\in\{1,...,d\}$ instead of distinguishing between differently sized multi-indices for every $l$ that appears. The zeros at the end substitutionally indicate real numbers. They commute with every multivector. That implies, that for the last $d-l$ factors no swap and therefore no separation needs to be made. It would also be possible to use the notation $\c^{(J)_l}(f_1,...,f_{l-1},0,...,0)$ for $l\in\{1,...,d\}$, because every function commutes with itself. We chose the other one where no exceptional treatment of $f_1$ is necessary. But please note that the multivectors $(J)_l$ indicating the commutative and anticommutative parts will all have zeros from $l$ to $d$ and therefore form a strictly triangular matrix. %
% \begin{defn}
% For a multi-index ${\j\in\{0,1\}^{\mu}}$ we define the set of strict lower triangular matrices $M(\j)$
% \end{defn}
%
\begin{lem}\label{l:prodtrennen}
 Let a set of functions $F=\{f_1(\x,\u),...,f_{d}(\x,\u)\}$ fulfill (\ref{ifix}) and be linear with respect to $\x$. Further let $J\in\{0,1\}^{d\times d}$ be a strictly lower triangular matrix, that is associated column by column with a multi-index ${\j\in\{0,1\}^{d}}$ by $\forall k=1,...,d:%$ let the sum over the $k$-th column $
(\sum_{l=1}^{d}J_{l,k})\bmod 2=j_{k}$, with $(J)_l$ being its $l$-th row, then 
\begin{eqnarray}\begin{aligned}
&\prod_{l=1}^d e^{-f_l(\x+\y,\u)}
\\&=\sum_{\j\in\{0,1\}^{d}}\sum_{\stackrel{J\in\{0,1\}^{d\times d},}{\sum_{l=1}^d(J)_l\bmod 2=\j}}\prod_{l=1}^{d} e^{-f_l(\x,\u)}_{\c^{(J)_l}(\overleftarrow{f_1,...,f_l,0,...,0})}\prod_{l=1}^{d}e^{-(-1)^{j_l}f_l(\y,\u)}
\end{aligned}\end{eqnarray}
or alternatively with strictly upper triangular matrices $J$
\begin{eqnarray}\begin{aligned}
&\prod_{l=1}^d e^{-f_l(\x+\y,\u)}
\\&=\sum_{\j\in\{0,1\}^{d}}\sum_{\stackrel{J\in\{0,1\}^{d\times d},}{\sum_{l=1}^d(J)_l\bmod 2=\j}}\prod_{l=1}^{d}e^{-(-1)^{j_l}f_l(\x,\u)}\prod_{l=1}^{d} e^{-f_l(\y,\u)}_{\c^{(J)_l}(\overrightarrow{0,...,0,f_l,...,f_d})}.
\end{aligned}\end{eqnarray}
\end{lem}
We do not explicitly indicate the dependence of the partition on $\u$ as in corollary \ref{k:expcom}, because the functions in the exponents already contain this dependence. Please note that the decomposition is pointwise.
\begin{proof} We will only prove the first assertion. The second one follows analogously by applying corollary \ref{k:expcom} the other way around. %From the linearity of the functions and Lemma \ref{l:potenzgehiftetz} we get
\begin{eqnarray}\begin{aligned}
\prod_{l=1}^d e^{-f_l(\x+\y,\u)}
\overset{F\text{ lin.}}=&\prod_{l=1}^d e^{-f_l(\x,\u)-f_l(\y,\u)}\\
\overset{\text{Lem. }\ref{l:potenzgesetz}}=&
\prod_{l=1}^d e^{-f_l(\x,\u)}e^{-f_l(\y,\u)}
% \end{aligned}\end{eqnarray}
% Now we use corollary \ref{k:expcom} to step by step rearrange the order of the product.
% \begin{eqnarray}\begin{aligned}
\\=&e^{-f_1(\x,\u)}e^{-f_1(\y,\u)}\prod_{l=2}^d e^{-f_l(\x,\u)}e^{-f_l(\y,\u)}\\
% \overset{\text{Lem. }\ref{l:expcom}}=&e^{-f_1(\x,\u)}e^{-f_2(\x,\u)}_{\c^{0}(f_{1})} e^{-f_1(\x,\u)} e^{-f_2(\x,\u)}\prod_{l=3}^d e^{-f_l(\x,\u)}e^{-f_l(\y,\u)}\\
% &+e^{-f_1(\x,\u)}e^{-f_2(\x,\u)}_{\c^{1}(f_{1})} e^{f_1(\x,\u)}e^{-f_2(\x,\u)}\prod_{l=3}^d e^{-f_l(\x,\u)}e^{-f_l(\y,\u)}\\
\overset{\text{cor. }\ref{k:expcom}}=&e^{-f_1(\x,\u)}(e^{-f_2(\x,\u)}_{\c^{0}(f_{1})} e^{-f_1(\y,\u)} e^{-f_2(\y,\u)}
\\&+e^{-f_2(\x,\u)}_{\c^{1}(f_{1})} e^{f_1(\y,\u)}e^{-f_2(\y,\u)})\prod_{l=3}^d e^{-f_l(\x,\u)}e^{-f_l(\y,\u)}\\
 \end{aligned}\end{eqnarray}
% Now we use corollary \ref{k:expcom} to step by step rearrange the order of the product.
 \begin{eqnarray}\begin{aligned}
\overset{\text{cor. }\ref{k:expcom}}=&e^{-f_1(\x,\u)}
(e^{-f_2(\x,\u)}_{\c^{0}(f_{1})}e^{-f_3(\x,\u)}_{\c^{00}(\overleftarrow{f_{1},f_2})} e^{-f_1(\y,\u)} e^{-f_2(\y,\u)}e^{-f_3(\y,\u)}\\&
+e^{-f_2(\x,\u)}_{\c^{0}(f_{1})}e^{-f_3(\x,\u)}_{\c^{01}(\overleftarrow{f_{1},f_2})} e^{-f_1(\y,\u)} e^{f_2(\y,\u)}e^{-f_3(\y,\u)}\\&
+e^{-f_2(\x,\u)}_{\c^{0}(f_{1})}e^{-f_3(\x,\u)}_{\c^{10}(\overleftarrow{f_{1},f_2})} e^{f_1(\y,\u)} e^{-f_2(\y,\u)}e^{-f_3(\y,\u)}\\&
+e^{-f_2(\x,\u)}_{\c^{0}(f_{1})}e^{-f_3(\x,\u)}_{\c^{11}(\overleftarrow{f_{1},f_2})} e^{f_1(\y,\u)} e^{f_2(\y,\u)}e^{-f_3(\y,\u)}\\&
+e^{-f_2(\x,\u)}_{\c^{1}(f_{1})}e^{-f_3(\x,\u)}_{\c^{00}(\overleftarrow{f_{1},f_2})} e^{f_1(\y,\u)} e^{-f_2(\y,\u)}e^{-f_3(\y,\u)}\\&
+e^{-f_2(\x,\u)}_{\c^{1}(f_{1})}e^{-f_3(\x,\u)}_{\c^{01}(\overleftarrow{f_{1},f_2})} e^{f_1(\y,\u)} e^{f_2(\y,\u)}e^{-f_3(\y,\u)}\\&
+e^{-f_2(\x,\u)}_{\c^{1}(f_{1})}e^{-f_3(\x,\u)}_{\c^{10}(\overleftarrow{f_{1},f_2})} e^{-f_1(\y,\u)} e^{-f_2(\y,\u)}e^{-f_3(\y,\u)}\\&
+e^{-f_2(\x,\u)}_{\c^{1}(f_{1})}e^{-f_3(\x,\u)}_{\c^{11}(\overleftarrow{f_{1},f_2})} e^{-f_1(\y,\u)} e^{f_2(\y,\u)}e^{-f_3(\y,\u)})
\\&\prod_{l=4}^d e^{-f_l(\x,\u)}e^{-f_l(\y,\u)}
\end{aligned}\end{eqnarray}
There are only $2^\delta$ ways of distributing the signs of $\delta$ exponents, so some of the summands can be combined. 
\begin{eqnarray}\begin{aligned}
=&e^{-f_1(\x,\u)}
((e^{-f_2(\x,\u)}_{\c^{0}(f_{1})}e^{-f_3(\x,\u)}_{\c^{00}(\overleftarrow{f_{1},f_2})}+e^{-f_2(\x,\u)}_{\c^{1}(f_{1})}e^{-f_3(\x,\u)}_{\c^{10}(\overleftarrow{f_{1},f_2})})
\\& e^{-f_1(\y,\u)} e^{-f_2(\y,\u)}e^{-f_3(\y,\u)}\\&
+(e^{-f_2(\x,\u)}_{\c^{0}(f_{1})}e^{-f_3(\x,\u)}_{\c^{01}(\overleftarrow{f_{1},f_2})}+e^{-f_2(\x,\u)}_{\c^{1}(f_{1})}e^{-f_3(\x,\u)}_{\c^{11}(\overleftarrow{f_{1},f_2})})
\\& e^{-f_1(\y,\u)} e^{f_2(\y,\u)}e^{-f_3(\y,\u)}\\&
+(e^{-f_2(\x,\u)}_{\c^{0}(f_{1})}e^{-f_3(\x,\u)}_{\c^{10}(\overleftarrow{f_{1},f_2})}+e^{-f_2(\x,\u)}_{\c^{1}(f_{1})}e^{-f_3(\x,\u)}_{\c^{00}(\overleftarrow{f_{1},f_2})} e^{f_1(\y,\u)}
\\& e^{-f_2(\y,\u)}e^{-f_3(\y,\u)}\\&
+(e^{-f_2(\x,\u)}_{\c^{0}(f_{1})}e^{-f_3(\x,\u)}_{\c^{11}(\overleftarrow{f_{1},f_2})}+e^{-f_2(\x,\u)}_{\c^{1}(f_{1})}e^{-f_3(\x,\u)}_{\c^{01}(\overleftarrow{f_{1},f_2})}) e^{f_1(\y,\u)}
\\& e^{f_2(\y,\u)}e^{-f_3(\y,\u)})
\prod_{l=4}^d e^{-f_l(\x,\u)}e^{-f_l(\y,\u)}
\end{aligned}\end{eqnarray}
To get a compact notation we expand all multi-indices by adding zeros until they have the same length. Note that the last non zero argument in terms like $\c^{000}(\overleftarrow{f_{1},0,0})$ always coincides with the exponent of the corresponding factor. Because of that it will always commute and could as well be replaced by a zero, too.
\begin{eqnarray}\begin{aligned}
 =&e^{-f_1(\x,\u)}_{\c^{000}(\overleftarrow{f_{1},0,0})}\\&
((e^{-f_2(\x,\u)}_{\c^{000}(\overleftarrow{f_{1},f_2,0})}e^{-f_3(\x,\u)}_{\c^{000}(\overleftarrow{f_{1},f_2,f_3})}+e^{-f_2(\x,\u)}_{\c^{100}(\overleftarrow{f_{1},f_2,0})}e^{-f_3(\x,\u)}_{\c^{100}(\overleftarrow{f_{1},f_2,f_3})})
\\& e^{-f_1(\y,\u)} e^{-f_2(\y,\u)}e^{-f_3(\y,\u)}\\&
+(e^{-f_2(\x,\u)}_{\c^{000}(\overleftarrow{f_{1},f_2,0})}e^{-f_3(\x,\u)}_{\c^{010}(\overleftarrow{f_{1},f_2,f_3})}+e^{-f_2(\x,\u)}_{\c^{100}(\overleftarrow{f_{1},f_2,0})}e^{-f_3(\x,\u)}_{\c^{110}(\overleftarrow{f_{1},f_2,f_3})})
\\& e^{-f_1(\y,\u)} e^{f_2(\y,\u)}e^{-f_3(\y,\u)}\\&
+(e^{-f_2(\x,\u)}_{\c^{000}(\overleftarrow{f_{1},f_2,0})}e^{-f_3(\x,\u)}_{\c^{100}(\overleftarrow{f_{1},f_2,f_3})}+e^{-f_2(\x,\u)}_{\c^{100}(\overleftarrow{f_{1},f_2,0})}e^{-f_3(\x,\u)}_{\c^{000}(\overleftarrow{f_{1},f_2,f_3})}
\\& e^{f_1(\y,\u)} e^{-f_2(\y,\u)}e^{-f_3(\y,\u)}\\&
+(e^{-f_2(\x,\u)}_{\c^{000}(\overleftarrow{f_{1},f_2,0})}e^{-f_3(\x,\u)}_{\c^{110}(\overleftarrow{f_{1},f_2,f_3})}+e^{-f_2(\x,\u)}_{\c^{100}(\overleftarrow{f_{1},f_2,0})}e^{-f_3(\x,\u)}_{\c^{010}(\overleftarrow{f_{1},f_2,f_3})})
\\& e^{f_1(\y,\u)} e^{f_2(\y,\u)}e^{-f_3(\y,\u)})\\&
\prod_{l=4}^d e^{-f_l(\x,\u)}e^{-f_l(\y,\u)}
\end{aligned}\end{eqnarray}
For $\delta=3$ we look at all strictly lower triangular matrices $J\in\{0,1\}^{\delta\times \delta}$ with the property 
\begin{eqnarray}\begin{aligned}\label{Jproperty}
\forall k=1,...,\delta:(\sum_{l=1}^\delta(J)_{l,k})\bmod 2=j_{k}.
\end{aligned}\end{eqnarray}
 That means the $l$-th row $(J)_{l}$ of $J$ contains a multi-index $(J)_{l}\in\{0,1\}^{\delta}$, with the last $\delta-l-1$ entries being zero and the $k$-th column sum being even when $j_k=0$ and being odd when $j_k=1$. For example the first multi-index is $\j=(0,0,0)$. There are only two different strictly lower triangular matrices that have columns summing up to even numbers: 
\begin{eqnarray}\begin{aligned}
J=\begin{pmatrix}0&0&0\\0&0&0\\0&0&0\end{pmatrix} \text{ and } J=\begin{pmatrix}0&0&0\\1&0&0\\1&0&0\end{pmatrix}.
\end{aligned}\end{eqnarray} 
Their first row contains the multi-index that belongs to $e^{-f_1(\x,\u)}$, the second one belongs to $e^{-f_2(\x,\u)}$ and so on. So the summands with exactly these multi-indices are the ones assigned to the product of exponentials whose signs are invariant during the reordering. With this notation and all $J\in\{0,1\}^{3\times3}$ that satisfy the property (\ref{Jproperty}) we can write
\begin{eqnarray}\begin{aligned}
\prod_{l=1}^d e^{-f_l(\x+\y,\u)}=&\sum_{\j\in\{0,1\}^{3}}\sum_{J}\prod_{l=1}^3 e^{-f_l(\x,\u)}_{\c^{(J)_l}(\overleftarrow{f_{1},...,f_l,0,...,0})} \prod_{l=1}^3e^{-(-1)^{j_l}f_l(\y,\u)}
\\& \prod_{l=4}^d e^{-f_l(\x,\u)}e^{-f_l(\y,\u)}.
\end{aligned}\end{eqnarray}
Using mathematical induction with matrices $J\in\{0,1\}^{\delta\times \delta}$ like introduced above for growing $\delta$ and corollary \ref{k:expcom} repeatedly until we reach $\delta=d$ we get
\begin{eqnarray}\begin{aligned}
 =&\sum_{\j\in\{0,1\}^{d}}\sum_{J}\prod_{l=1}^{d} e^{-f_l(\x,\u)}_{\c^{(J)_l}(\overleftarrow{f_{1},...,f_l,0,...,0})} \prod_{l=1}^{d}e^{-(-1)^{j_l}f_l(\y,\u)}.
\end{aligned}\end{eqnarray}
\end{proof}
\begin{rem}\label{b:upperlimit}
The number of actually appearing summands is usually much smaller than in Theorem \ref{t:shift}. It is determined by the amount of distinct strictly lower (upper) triangular matrices $J$ with entries being either zero or one, particularly
\begin{eqnarray}\begin{aligned}\label{upperlimit}
 2^{\frac{d(d-1)}2}.
\end{aligned}\end{eqnarray}
% In practical applications $\nu$ is normally not higher than 4.
\end{rem}
%
%---------------------------------------------------------------------------------------------------------------------------------
%
\begin{thm}[Shift]\label{t:shift}\index{geometric Fourier transform!shift theorem}
Let $\A(\x)=\B(\x-\x_0)$ be multivector fields, $F_1,F_2$ be linear with respect to $\x$, ${\j\in\{0,1\}^{\mu}},\k\in\{0,1\}^{(\nu-\mu)}$ be multi-indices and $F_1(\j),F_2(\k)$ as introduced in Definition \ref{d:fj}% Further let $J\in\{0,1\}^{\mu\times \mu}$ be the strictly lower triangular and $K\in\{0,1\}^{(\nu-\mu) \times (\nu-\mu)}$ the strictly upper triangular matrices with rows $(J)_{l},(K)_{l}$ summing up to $(\sum_{l=1}^{\mu}(J)_{l})\bmod 2=\j$ respectively $(\sum_{l=\mu+1}^{\nu}(K)_{l-\mu})\bmod 2=\k$ as in Lemma \ref{l:prodtrennen}
, then a separable GFT suffices
% \begin{eqnarray}\begin{aligned}
% \F_{F_1,F_2}(\A)(\u)=&\sum_{\j,\k}\sum_{J,K}\prod_{l=1}^{\mu} e^{-f_l(\x_0,\u)}_{\c^{(J)_l}(\overleftarrow{f_{1},...,f_l,0,...0})} \F_{F_1(\j),F_2(\k)}(\B)(\u)\\&\prod_{l=\mu+1}^{\nu} e^{-f_l(\x_0,\u)}_{\c^{(K)_{l-\mu}}(\overleftarrow{f_{\mu+1},...,f_{l},0,...0})}
% \end{aligned}\end{eqnarray}
\begin{eqnarray}\begin{aligned}
\F_{F_1,F_2}(\A)(\u)=&\sum_{\j,\k}\sum_{J,K}\prod_{l=1}^{\mu} e^{-f_l(\x_0,\u)}_{\c^{(J)_l}(\overleftarrow{f_{1},...,f_l,0,...,0})} \F_{F_1(\j),F_2(\k)}(\B)(\u)
\\& \prod_{l=\mu+1}^{\nu} e^{-f_l(\x_0,\u)}_{\c^{(K)_{l-\mu}}(\overrightarrow{0,...,0,f_{l},...,f_{\nu}})},
\end{aligned}\end{eqnarray}
where $J\in\{0,1\}^{\mu\times \mu}$ and $K\in\{0,1\}^{(\nu-\mu) \times (\nu-\mu)}$ are the strictly lower, respectively upper, triangular matrices with rows $(J)_{l},(K)_{l-\mu}$ summing up to $(\sum_{l=1}^{\mu}(J)_{l})\bmod 2=\j$ respectively $(\sum_{l=\mu+1}^{\nu}(K)_{l-\mu})\bmod 2=\k$ as in Lemma \ref{l:prodtrennen}.
\end{thm}
%
% \begin{rem}
% We can produce three more ways to furmulate Theorem \ref{t:shift} if we additionally use the alternative form of Lemma \ref{l:prodtrennen}. Depending on the particular GFT a certain form
% \end{rem}
\begin{proof}
First we put the transformed function down to $\B(\y)$ using a change of coordinates. 
\begin{eqnarray}\begin{aligned}
\F_{F_1,F_2}(\A)(\u)=&\int_{\R^m}\prod_{l=1}^{\mu} e^{-f_l(\x,\u)}\A(\x)\prod_{l=\mu+1}^{\nu}e^{-f_l(\x,\u)}\d^m\x\\
=&\int_{\R^m}\prod_{l=1}^{\mu} e^{-f_l(\x,\u)}\B(\x-\x_0)\prod_{l=\mu+1}^{\nu}e^{-f_l(\x,\u)}\d^m\x\\
\overset{\y=\x-\x_0}=&\int_{\R^m}\prod_{l=1}^{\mu} e^{-f_l(\y+\x_0,\u)}\B(\y)\prod_{l=\mu+1}^{\nu}e^{-f_l(\y+\x_0,\u)}\d^m\y
\end{aligned}\end{eqnarray}
Now we separate and sort the factors with the above Lemma \ref{l:prodtrennen}.
\begin{eqnarray}\begin{aligned}
\overset{\text{ Lem. }\ref{l:prodtrennen}}=&\int_{\R^m}\sum_{\j\in\{0,1\}^{\mu}}\sum_{\stackrel{J\in\{0,1\}^{\mu\times \mu}}{\sum(J)_l\bmod 2=\j}}
\\& \prod_{l=1}^{\mu} e^{-f_l(\x_0,\u)}_{\c^{(J)_l}(\overleftarrow{f_{1},...,f_l,0,...,0})}\prod_{l=1}^{\mu}e^{-(-1)^{j_l}f_l(\y,\u)}\B(\y)\\
&\sum_{\k\in\{0,1\}^{(\nu-\mu)}}\sum_{\stackrel{K\in\{0,1\}^{(\nu-\mu) \times (\nu-\mu)}}{\sum(K)_l\bmod 2=\k}}
\\& \prod_{l=\mu+1}^{\nu}e^{-(-1)^{k_{l-\mu}}f_l(\y,\u)}\prod_{l=\mu+1}^{\nu} e^{-f_l(\x_0,\u)}_{\c^{(K)_{l-\mu}}(\overrightarrow{0,...,0,f_{l},...,f_{\nu}})}\d^m\y\\
=&\sum_{\j,\k}\sum_{J,K}\prod_{l=1}^{\mu} e^{-f_l(\x_0,\u)}_{\c^{(J)_l}(\overleftarrow{f_{1},...,f_l,0,...,0})}
\\& \F_{F_1(\j),F_2(\k)}(\B)(\u)\prod_{l=\mu+1}^{\nu} e^{-f_l(\x_0,\u)}_{\c^{(K)_{l-\mu}}(\overrightarrow{0,...,0,f_{l},...,f_{\nu}})}
\end{aligned}\end{eqnarray}
\end{proof}
\begin{cor}[Shift]\label{k:shift}
 Let $\A(\x)=\B(\x-\x_0)$ be multivector fields, $F_1$ and $F_2$ each consist of mutually commutative functions\footnote{Cross commutativity between $F_1$ and $F_2$ is not necessary.} being linear with respect to $\x$, then the GFT obeys
\begin{eqnarray}\begin{aligned}
\F_{F_1,F_2}(\A)(\u)=&\prod_{l=1}^{\mu} e^{-f_l(\x_0,\u)} \F_{F_1,F_2}(\B)(\u)\prod_{l=\mu+1}^{\nu} e^{-f_l(\x_0,\u)}.
\end{aligned}\end{eqnarray}
\end{cor}
\begin{rem}
 For sets $F_1,F_2$ that each consist of less than two functions the condition of corollary \ref{k:shift} is necessarily satisfied, compare e. g. reference examples %\ref{bspgft} 
1,3 and 4.
\end{rem}
 The specific forms, our standard examples %(Example \ref{bspgft}) 
take, are summarized in Table 3. As expected they are often shorter than what could be expected from Remark \ref{b:upperlimit}.%the theoretical upper limit guarantees.
\par
%\vspace{12pt}
\begin{table}[hbt!]
\begin{tabular}{|l|l|l|}
\hline
&GFT		& $\A(\x)=\B(\x-\x_0)$	 	\rule [-1.8mm]{0mm}{6mm}\\\hline
1.& Clifford 	& $\F_{f_1}= \F_{f_1} e^{-2\pi i\x_0\cdot\u}$ \rule [-1.8mm]{0mm}{6mm}\\
2.& B\"ulow		& $\F_{f_1,...,f_n}=\sum_{\k\in\{0,1\}^n}\sum_{K}\F_{(-1)^{k_1}f_1,...,(-1)^{k_n}f_n}$
\rule [-1.8mm]{0mm}{6mm}\\&&\phantom{$\F_{f_1,...,f_n}=$}$
\prod_{l=1}^{n} e^{-2\pi {x_0}_ku_k}_{\c^{(K)_{l}}(\overrightarrow{0,...,0,f_{l},...,f_{n}})}$ 			\rule [-1.8mm]{0mm}{6mm}\\
3.& Quaternionic	& $\F_{f_1,f_2}= e^{-2\pi i{x_0}_1u_1}\F_{f_1,f_2} e^{-2\pi j{x_0}_2u_2}$	 \rule [-1.8mm]{0mm}{6mm}\\
4.& Spacetime	& $\F_{f_1,f_2}= e^{-\e_4{x_0}_4u_4}\F_{f_1,f_2}e^{-\epsilon_4\e_4i_4(x_1u_1+x_2u_2+x_3u_3)}$	 \rule [-1.8mm]{0mm}{6mm}\\
5.& Color Image	&$\F_{f_1,f_2,f_3,f_4}=e^{-\frac12({x_0}_1u_1+{x_0}_2u_2)(\B+i\B)}\F_{f_1,f_2,f_3,f_4}$ \rule [-1.8mm]{0mm}{6mm}\\&&\phantom{$\F_{f_1,f_2,f_3,f_4}=$}$
e^{\frac12({x_0}_1u_1+{x_0}_2u_2)(\B+i\B)}$	\rule [-1.8mm]{0mm}{6mm}\\
6.& Cyl. $n=2$	& $\F_{f_1}=e^{\x_0\wedge\u}\F_{f_1}$				\rule [-1.8mm]{0mm}{6mm}\\
& Cyl. $n\neq2$	& -				\rule [-1.8mm]{0mm}{6mm}\\\hline
\end{tabular}%\label{tab:g2}
\caption{Theorem \ref{t:shift} (Shift) applied to the GFTs of the first example%\ref{bspgft}
, enumerated in the same order. Notations: on the LHS $\F_{F_1,F_2}=\F_{F_1,F_2}(\A)(\u)$, on the RHS $\F_{F^{\prime}_1,F^{\prime}_2}=\F_{F^{\prime}_1,F^{\prime}_2}(\B)(\u)$, in the second row $K$ represents all strictly upper triangular matrices $\in\{0,1\}^{n \times n}$ with rows $(K)_{l-\mu}$ summing up to $(\sum_{l=\mu+1}^{\nu}(K)_{l-\mu})\bmod 2=\k$. The simplified shape of the color image FT results from the commutativity of $\B$ and $i\B$ and application of Lemma \ref{l:potenzgesetz}.}
\end{table}
%
%
%---------------------------------------------------------------------------------------------------------------------------------
\section{Conclusions and Outlook}
%---------------------------------------------------------------------------------------------------------------------------------
%
For multivector fields over $\R^{p,q}$ with values in any geometric algebra $G^{p,q}$ we have successfully defined a general geometric Fourier transform. It comprehends all popular Fourier transforms from current literature in the introductory example% \ref{bspgft}
.
Its existence, independent from the specific choice of functions $F_1,F_2$, could be proved for all integrable multivector fields, see Theorem \ref{t:ex}.
Theorem \ref{t:lin} shows that our geometric Fourier transform is generally linear over the field of real numbers.
All transforms from the reference example %\ref{bspgft} 
consist of bilinear $F_1$ and $F_2$. We proved that this property is sufficient to ensure the scaling property of Theorem \ref{t:scaling}. 
\par
If a general geometric Fourier transform is separable as introduced in Definition \ref{d:sep}, then Theorem \ref{t:products} (Left and right products) guarantees that constant factors can be separated from the vector field to be transformed. As a consequence general linearity is achieved by choosing $F_1,F_2$ with values in the center of the geometric algebra $\clifford{p,q}$, compare Corollary \ref{k:lin}. All examples except for the cylindrical Fourier transform \cite{BSS10} satisfy this claim.
\par
Under the condition of linearity with respect to the first argument of the functions of the sets $F_1$ and $F_2$ additionally to the just mentioned separability, we also proved a shift property (Theorem \ref{t:shift}).
\par
% \par
% %\vspace{12pt}
% \begin{table}[hbt!]
%  
% \begin{tabular}{|l|l|}
% \hline
% GFT		& $\A(\x)=\B(\x-\x_0)$	 	\rule [-1.8mm]{0mm}{6mm}\\\hline
% Clifford 	& $\F_{f_1}= \F_{f_1} e^{-f_1(\x_0,\u)}$ \rule [-1.8mm]{0mm}{6mm}\\
% B\"ulow		& $\F_{f_1,...,f_n}=\sum_{\k\in\{0,1\}^n}\sum_{K}\F_{(-1)^{k_1}f_1,...,(-1)^{k_n}f_n}\prod_{l=1}^{n} e^{-f_l(\x_0,\u)}_{\c^{(K)_{l}}(\overrightarrow{0,...,0,f_{l},...,f_{n}})}$ 			\rule [-1.8mm]{0mm}{6mm}\\
% Quaternionic	& $\F_{f_1,f_2}= e^{-f_1(\x_0,\u)}\F_{f_1,f_2} e^{-f_2(\x_0,\u)}$	 \rule [-1.8mm]{0mm}{6mm}\\
% Spacetime	& $\F_{f_1,f_2}= e^{-f_1(\x_0,\u)}\F_{f_1,f_2}e^{-f_2(\x_0,\u)}$	 \rule [-1.8mm]{0mm}{6mm}\\
% Color Image	&$\F_{f_1,f_2,f_3,f_4}=e^{-f_1(\x_0,\u)}e^{-f_2(\x_0,\u)}\F_{f_1,f_2,f_3,f_4}e^{-f_3(\x_0,\u)}e^{-f_4(\x_0,\u)}$	\rule [-1.8mm]{0mm}{6mm}\\
% Cylindrical	& -				\rule [-1.8mm]{0mm}{6mm}\\\hline
% \end{tabular}%\label{tab:g2}
% \caption{Theorem \ref{t:shift} (Shift) applied to examples \ref{bspgft}. Notations: on the LHS $\F_{F_1,F_2}=\F_{F_1,F_2}(\A)(\u)$, on the RHS $\F_{F^{\prime}_1,F^{\prime}_2}=\F_{F^{\prime}_1,F^{\prime}_2}(\B)(\u)$, in the second row $K$ represents all strictly upper triangular matrices $\in\{0,1\}^{n \times n}$ with rows $(K)_{l-\mu}$ summing up to $(\sum_{l=\mu+1}^{\nu}(K)_{l-\mu})\bmod 2=\k$.}
%  
% \end{table}
In future papers we are going to state the necessary constraints for a generalized convolution theorem, invertibility, derivation theorem and we will examine how simplifications can be achieved based on symmetry properties of the multivector fields to be transformed. We will also construct generalized geometric Fourier transforms in a broad sense from combinations of the ones introduced in this paper and from decomposition into their sine and cosine parts which will also cover the vector and bivector Fourier transforms of \cite{BS09}. It would further be of interest to extend our approach to Fourier transforms defined on spheres or other non-Euclidean manifolds, to functions in the Schwartz space and to square-integrable functions.
 \addcontentsline{toc}{section}{References}
 \bibliographystyle{unsrt} 
 \bibliography{./Literaturverzeichnis}

% \end{document}
% \bibitem{test} A. B. C. Test, \textit{On a Test.} J. of Testing
% \textbf{88} (2000), 100--120.
% \bibitem{latex} G. Gr\"atzer, \textit{Math into \LaTeX.} 3rd Edition,
% Birkh\"auser, 2000.
% \end{thebibliography}

%\printindex
% ------------------------------------------------------------------------
\end{document}